\documentclass{article}

\newcommand{\vekk}[1]{}
\usepackage[utf8]{inputenc}
\newcommand{\Renyi}{Rényi }
\input{def.tex}

\begin{document}


\title{Statistics with improper posteriors}
  
\author{Gunnar Taraldsen, Jarle Tufto, and Bo H. Lindqvist \\
  \\
  Department of Mathematical Sciences \\
  Norwegian University of Science and Technology
}


\maketitle

\begin{abstract}

  In 1933 Kolmogorov constructed a general theory that
  defines the modern concept of conditional probability.
  In 1955 Rényi fomulated a new axiomatic theory for probability motivated
  by the need to include unbounded measures.
  We introduce a general concept of conditional probability in Rényi spaces.  
  In this theory improper priors are allowed,
  and the resulting posteriors can also be improper.
  
  \vekk{
  In 1965 Lindley published his classic text on Bayesian statistics using the theory of Rényi,
  but retracted this idea in 1973 due to the appearance of marginalization paradoxes presented
  by Dawid, Stone, and Zidek.
  The paradoxes are investigated, and the seemingly conflicting results are explained.
  The theory of Rényi can hence be used as an axiomatic basis for statistics that
  allows use of unbounded priors.
  }
  
\end{abstract}


\tableofcontents

\newpage

\section{Introduction}
\label{sIntroduction}

An often voiced criticism of the use of improper priors in Bayesian inference 
is that such priors sometimes don't lead to a proper posterior distribution.
This can happen when the marginal law of the data $X$ is not
$\sigma$-finite \citep{TaraldsenLindqvist10ImproperPriors},
as sometimes encountered in applied settings with sparse data
\citep{BordBiocheDruilhet18,TuftoEtAl12butterflies}.

The dangers of improper posteriors in Markov Chain Monte Carlo based methods
of inference are well recognized \citep[e.g.][]{HobertCasella96}.
Within the theory to be presented here,
improper posteriors as such are well-defined, however,
and in practical applied statistics it will be of interest to
develop numerical methods for computing such posterior densities.
One possible method is indicated by \citet[Appendix S4]{TuftoEtAl12butterflies} 
for data on tropical butterflies,
and is illustrated in Fig~\ref{figImproperPosterior}.
\begin{figure}
    \centering
    \includegraphics[width=.8\textwidth]{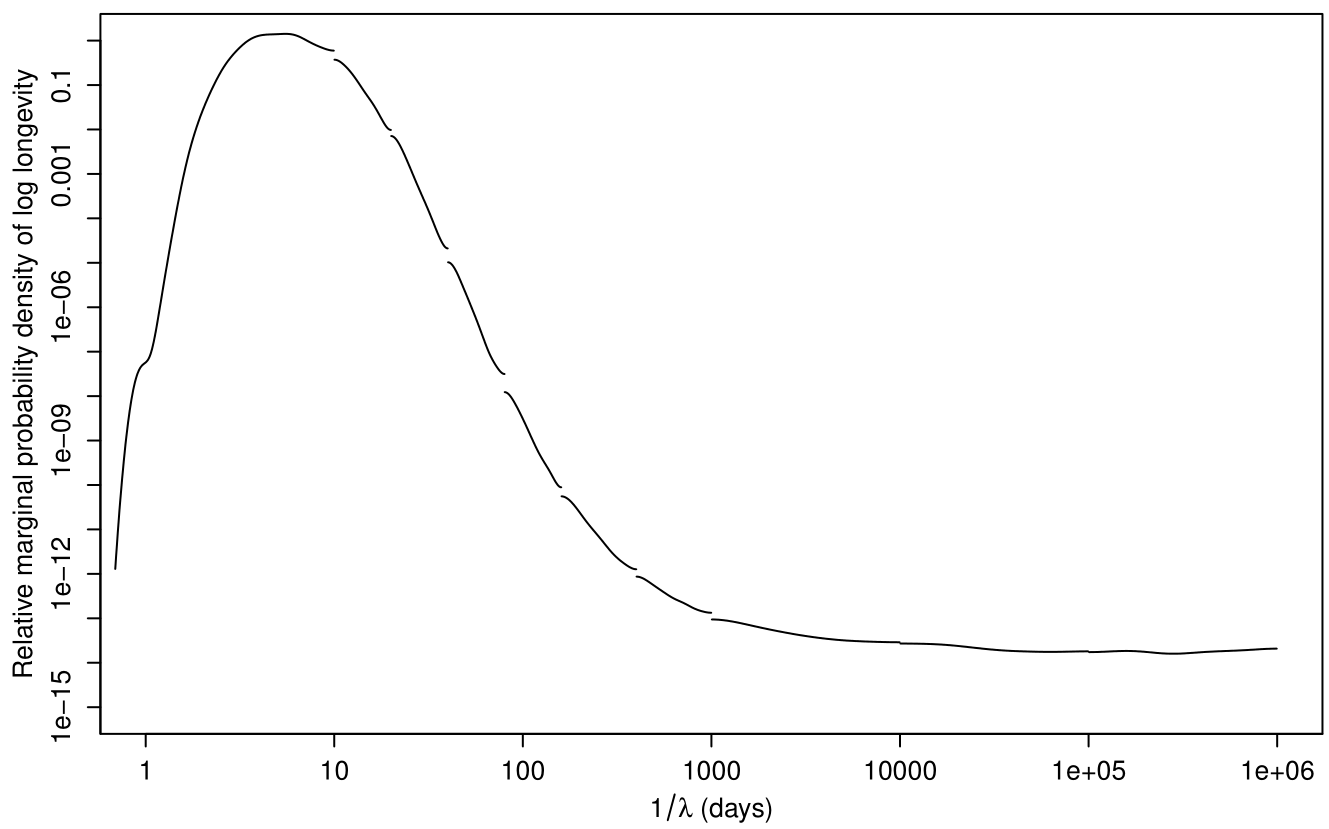}
    \caption{An estimate of the log of an improper posterior density.
      It is obtained by alignment of kernel density estimates based on separate MCMC run,
      each run restricted to different subintervals.} 
    \label{figImproperPosterior}
\end{figure}
The key idea is to consider the family of posteriors obtained from restriction to intervals,
and then glue the resulting posteriors together in a postprocessing step.
This simple idea is also the key for the general
definition of the posterior we introduce in Section~\ref{sTheory}.
The definition is based on the family of conditional probabilities
appearing in the axioms of
a conditional probability space as introduced by \citet{Renyi55axioms}.

As a simpler motivating example,
suppose you observe a homogeneous Poisson process with a scale invariant prior density
\citep[p.122]{JEFFREYS}
\be{1}
\pi (\lambda) = \frac{c}{\lambda}
\ee
on the Poisson intensity $\lambda$.
The constant $c$ is arbitrary, carries no information,
and $c=1$ is used below.
Similar arbitrary constants will, however, play an important role in
the theory in later parts of this paper.
The marginal law of the number $X$ of events in the interval $(0,t]$
is then not $\sigma$-finite since
\be{2}
P(X = 0) =\int_0^\infty P(X=0 | \lambda) \pi(\lambda)\, d \lambda
= \int_0^\infty \frac{(\lambda t)^0}{0!} e^{-\lambda t} \, \frac{d\lambda}{\lambda} = \infty
\ee
If you observe $X=0$ and formally multiply
the prior by the likelihood you obtain an improper posterior
\be{3}
\pi (\lambda | X=0) = \frac{e^{-\lambda t}}{\lambda}
\ee
This posterior law for $\lambda$ is different from the initial prior law, 
and we claim that this is a correct way of incorporating the
information given by $X=0$.
High values for $\lambda$ are less probable given the observation $X=0$.
Further updating can be done with this posterior as a prior,
and this is consistent with only one updating based on the initial prior.

A related example 
is the Beta posterior density 
for the success probability $p$ given by 
\be{4}
\pi (p \st x) = p^{x - 1} (1-p)^{n - x - 1}
\ee
for a Bernoulli sequence with $x$ successes out of $n$ trials.
This corresponds to the improper \citet{Haldane32prior} prior \citep[p.123]{JEFFREYS}  
\be{4a}
\pi (p) = p^{- 1} (1-p)^{- 1}
\ee
The posterior is improper if $x$ is zero as in the previous example.
In all cases, however, 
the observation of the number of successes $x$ results in a 
corresponding updating of the uncertainty associated with $p$.
The posterior is in this case improper for 
$x=0$ and for $x=n$.
In all cases, however, the posterior in equation~(\ref{eq4}) contains the
information given by the observation $x$ and 
the prior in equation~(\ref{eq4a}).

The Haldane prior is the prior that corresponds to the formal
Bayes estimator $\hat{p} = x/n$ \citep[p.29]{ROBERT}.
This is the optimal frequentist estimator
in the sense of being the unique uniformly minimum variance unbiased estimator. 
Similar optimality phenomena motivates the use of improper priors
more generally \citep[p.409]{BERGER}, and is also linked to fiducial inference
\citep{TaraldsenLindqvist13fidopt}.

Unfortunately, 
even people accepting the use of improper priors reject the above
form of inference,
on the ground that the posterior is not a probability distribution,
and a mathematical theory is lacking for this
\citep{RobertChopinRousseau09jeffreys}.
This is understandable, and we agree initially with this point of view. 
We will demonstrate, however, 
that the above forms of, so far,
formal inference can be made consistent with
the axiomatic system of Rényi which allows improper laws.
In Section~\ref{sTheory} we develop his mathematical theory
further to include general conditioning on a $\sigma$-field.
This gives a rigorous
mathematical foundation for inference
with unbounded laws - including the previous three examples.
The aim of this paper is to present key elements in a mathematical and philosophical  
theory of statistics that allows improper laws both as priors and posteriors
based on the concept of a \Renyi space $(\Omega, \ceps, \pr)$ as defined in the next section.

\vekk{
This extends the results of 
\citet{TaraldsenLindqvist10ImproperPriors,TaraldsenLindqvist16renyi},
and provides stronger links to the results on improper
laws presented by \citet{HARTIGAN} and \citet{BiocheDruilhet16convergence}.
The main mathematical result is Theorem~\ref{theo1}
in  Section~\ref{sTheory}
which proves existence and uniqueness of conditional expectation
on Rényi spaces.
The existence and uniqueness proof relies on the Radon-Nikodym theorem
and a generalization of the Rényi structure theorem.
Theorem~\ref{theo1} can alternatively  be seen as a generalisation of
the Radon-Nikodym theorem itself. 
\vekk{The theory of conditional expectation has been most important
for the development of measure theory, probability and statistics
based on Kolmogorov's concept of a probability space.
The generalization of this to the setting of Rényi spaces can hence
be expected to be important for future developments
in mathematics and related fields.
}

Section~\ref{sMathStat} presents some mathematical preliminaries for the
main results in Section~\ref{sTheory}
including in particular the definition of a Rényi space.
Section~\ref{sExamples} presents some examples,
including a discussion of the  marginalization paradoxes \citep{DawidStoneZidek73}.
Section~\ref{sDiscussion}
gives a discussion of concepts from a modelling
and more philosophical perspective.
The paper is rounded off with some final remarks
in Section~\ref{sRemarks}.
}

\section{Statistics in \Renyi space}
\label{sMathStat}
The mathematical theory of
statistics is thoroughly presented by \citet{SCHERVISH}.
We will next present the initial ingredients in this theory.
The purpose is to have a platform for a generalization
that effectively replaces the probability space of \citet{KOLMOGOROV} with
the conditional probability space of \citet{RENYI}.
The reader may feel that we include too
many elementary standard definitions, and we apologise for this.
The reason is that there are small differences in most books on foundations,
and the easiest way to be precise is to be explicit.
For measure theory we follow mostly the conventions in the elegant
treatment by \citet{RUDIN}.
A particular interpretation is indicated together with the mathematical theory,
but the reader should recognise that many other interpretations are possible.
We do not claim that there is one single ``correct'' interpretation,
but we do claim that the indicated philosophical
interpretation is useful in many applied concrete problems.

The initial ingredient is an abstract 
underlying space 
$\Omega$. 
This space $\Omega$ is a non-empty set equipped with a
law $\pr$ which assigns a weight 
$\pr (A)$ to all measurable sets $A \subset \Omega$.
The family $\ceps$ of measurable sets is assumed to be a $\sigma$-field:
(i) $\emptyset \in \ceps$,
(ii) $A \in \ceps$ implies $A^c \in\ceps$, and
(iii) $A_1, A_2, \ldots \in \ceps$ implies $\cup_i A_i \in \ceps$.
The set $\Omega$ equipped with the family $\ceps$ of measurable sets is then a measurable space
\citep[p.8]{RUDIN}.
A measurable set $A$ is also referred to as an event with
a corresponding philosophical interpretation \citep[p.1-37]{RENYI}.

The law $\pr$ is a positive measure defined on $\ceps$:
(i) $\pr (\emptyset) = 0$,
(ii) $A \in \ceps$ implies $0 \le \pr (A) \le \infty$,
(iii) If $A_1, A_2, \ldots \in \ceps$ are disjoint, then
$\pr (\sum_i A_i) = \sum_i \pr (A_i)$.
Property (iii) is referred to as countable additivity.
It is the distinctive feature that separates the theory here from the
alternative approach of including improper laws by allowing
finitely additive measures \citep{HeathSudderth89} \citep[p.21]{SCHERVISH}.
Additionally, in the theory of Kolmogorov, $\pr (\Omega) = 1$ is assumed.
The set $\Omega$ equipped with $\ceps$ and $\pr$ is then a probability space.
This is assumed in the following paragraphs until the concept of a \Renyi space is introduced.

\vekk{
We consider choice of notation to be important,
and are aware that the reader may have strong opinions on this,
but postpone a discussion of this until Section~\ref{sDiscussion}.
}

A random quantity $Z$
is a measurable function 
$Z: \Omega \into \Omega_Z$ \citep[p.583, p.606]{SCHERVISH}.
A function $Z$ is measurable if
$(Z \in A) = Z^{-1} (A) = \{\omega \st Z(\omega) \in A\}$ is measurable 
for any measurable $A \subset \Omega_Z$.
Using this notation,
the law $\pr_Z$ of $Z$ is well defined by
\be{LawZ}
\pr_Z (A) = \pr (Z \in A)
\ee
Another random quantity $W = \phi (Z)$ is defined by
$W (\omega) = \phi(Z(\omega))$ when
$\phi: \Omega_Z \into \Omega_W$ is measurable.
This implies $\pr_W = \pr_Z \circ \phi^{-1}$ so
the law 
of $W = \phi (Z)$
is determined by the law 
of $Z$.
The general change-of-variables theorem
$\E (\phi (Z)) = \E_Z (\phi)$ is also a consequence \citep[Thm B.12]{SCHERVISH}.
The notation $\E (W) = \int W(\omega)\, \pr (d\omega)$ is here used for the
expectation of $W$.
These observations explain partly why the abstract space $\Omega$ 
can be left unspecified in applications.

The previous paragraph defines the law $\pr_T$
of a random quantity $T$.
The notation
$\pr^t (A) = \pr (A \st T = t)$
is here used for the conditional law $\pr^t$ on $\Omega$. 
It is defined by the equation
\be{TheCondExp}
\pr ((T \in C) A) = \int_C \pr^t (A) \, \pr_T (dt)
\ee
The left-hand-side defines for each event $A$
a measure on $\Omega_T$
which is absolutely continuous with respect to $\pr_T$,
and $h(t) = \pr^t (A)$ is the unique density
$h \in L^1 (\pr_T)$ obtained from
the Radon-Nikodym theorem \citep[p.121]{RUDIN}.
The conditional law $\pr_Z^t$ of $Z$ given $T = t$ is defined by
\be{CondLawZ}
\pr_Z^t (A) = \pr^t (Z \in A) = \pr (Z \in A \st T=t)
\ee
The general change-of-variables theorem implies
that $\pr_Z^t (A) = (\pr_{Z,T})^t (A \times \Omega_T)$
so the conditional law $\pr_Z^t$ is determined
by the joint law $\pr_{Z,T}$.

The previous gives some basic ingredients from probability theory needed in a
mathematical theory of statistics.
For the theory of statistics \citet[p.82]{SCHERVISH} assumes,
as we also do,
that there is a single space $\Omega$
underlying also the statistical analysis
for a particular model $\theta$ with observed data $x$.
The data $X$ and the model $\Theta$ are random quantities.
This means that
$X: \Omega \into \Omega_X$ and $\Theta: \Omega \into \Omega_\Theta$
are measurable functions.
The data space $\Omega_X$ is the set corresponding
to the possible observations $x$.
It is commonly referred to as the sample space.
The model space $\Omega_\Theta$ is the set of
possible model parameters $\theta$.
It is sometimes referred to as the model parameter space.
A parameter $\Gamma = \psi (\Theta)$ is by definition a function of
the model $\Theta$.
It is hence also consistent to refer to $\Theta$ as the model parameter.
A statistic $Y = \phi (X)$ is by definition a function of the data $X$.
The function $\phi: \Omega_X \into \Omega_Y$ is sometimes referred to
as an action, and the set $\Omega_Y$ is then the action space.
The parameter $\Gamma$ is sometimes referred to as the focus parameter, and the set $\Omega_\Gamma$ is then the focus space.
These concepts are as explained in much more detail by \citet{SCHERVISH},
but with some differences in notation and naming conventions.
The involved concepts are illustrated in the commutative diagram in Figure~{\ref{figStatMod}}.
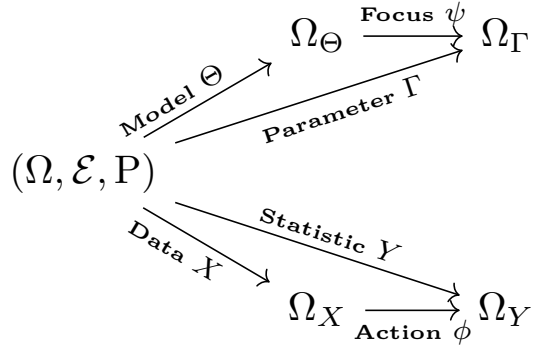
\begin{figure}

{
  \[
    \scalebox{1.5}{
  \begin{tikzcd}[row sep=normal, ampersand replacement=\&]
\&\Omega_\Theta
\arrow[r, "\text{\bf\tiny Focus } \psi"] \& \Omega_\Gamma  \\
(\Omega, {\cal E}, \pr)
\arrow[sloped, ur, "\text{\bf\tiny Model } \Theta" near end]
\arrow[sloped, drr, "\text{\bf\tiny Statistic } Y" near start]
\arrow[sloped,dr, "\text{\bf\tiny Data } X"' near end]
\arrow[sloped, urr, "\text{\bf\tiny Parameter } \Gamma"' near start] \& \& \\
\&\Omega_X
\arrow[sloped, r, "\text{\bf\tiny Action } \phi"'] \& \Omega_Y
\end{tikzcd}
}
\]
}
\caption{\label{figStatMod}
  A commutative diagram for a general statistical inference problem.}
\end{figure}

A statistical model for observed data $x$ is conventionally specified by
a family $\{\pr_X^\theta \st \theta \in \Omega_\Theta\}$
of probability measures $\pr_X^\theta$ indexed by the unknown model parameter $\theta$
\citep[p.1]{LehmannCasella98Estimation}.
We assume additionally, as does \citet[p.83]{SCHERVISH},
that the statistical model is given by the conditional law
\be{SchervishStatMod}
\pr_X^\theta (A) = \pr (X \in A \st \Theta = \theta)
\ee
of the data $X$ given $\Theta = \theta$. 
This requires also a specification of
the data space $\Omega_X$ and the model space $\Omega_\Theta$.
The task for the statistician is to infer
something about a chosen focus parameter
$\gamma = \psi(\theta) = \Gamma (\omega)$
from the observed data $x = X (\omega)$.
This is done by reporting a statistic $y = \phi (x)$.
The problem is then to choose or characterise a suitable action $\phi$,
and to implement and perform associated calculations. 

In Bayesian inference the prior $\pr_\Theta$ is also specified,
and together with the conditional law $\pr_X^\theta$
this determines the joint law of $X$ and $\Theta$.
The joint law of the data $X$ and the model $\Theta$
determines the
posterior law $\pr^x_\Theta$.
The simplicity and generality of this transformation of
prior knowledge $\pr_\Theta$ into the
posterior knowledge $\pr^x_\Theta$ given the data $x$
is one major argument in favour of the Bayesian paradigm. 
Additionally, it can be observed that the Bayes posterior expectation
$\phi(x) = \E^x (\Gamma)$ exemplify that the Bayes posterior
can be used to define many possible actions $\phi$ in addition
to the distribution estimator given by
the posterior law $\phi_D (x) = \pr_\Theta^x$ itself.

The previous paragraphs give a condensed presentation of some of
the initial ingredients in the well established mathematical theory of
statistics as presented in considerable more detail by
\citet{SCHERVISH}.
We now turn to the more general case where
$\Omega$ is a \Renyi space as will be defined in the next few paragraphs. 
Assume first that $(\Omega, \pr, \ceps)$ is a $\sigma$-finite measure space.
Let $\cB \subset \ceps$ denote the family of elementary conditions
$B$ defined by the requirement $0 < \pr (B) < \infty$.
A family $\{\pr (\cdot \st B) \st B \in \cB\}$ 
of conditional probability measures is then defined by
\be{RenyiCond}
\pr (A \st B) = \frac{\pr (A B)}{\pr(B)}, \; \forall A \in \ceps, \; \forall B \in \cB
\ee
It can be verified that 
$\cB$ is a bunch \citep[Def.2.2.1]{RENYI}:
(i) $\emptyset \not\in \cB$,
(ii) $B_1, B_2 \in \cB$ implies $B_1 \cup B_2 \in \cB$,
(iii) There exists a sequence $B_1, B_2, \ldots \in \cB$
with $\Omega = \cup_i B_i$.
Condition (iii) follows since $\pr$ is $\sigma$-finite.
Furthermore,
$B_1, B_2 \in \cB$ and $B_1 \subset B_2$ 
imply $\pr (B_1 \st B_2) > 0$, 
and imply also the consistency requirement
\be{RenyiSpace}
\pr (A \st B_1) = \frac{\pr (A B_1 \st B_2)}{\pr (B_1 \st B_2)}
\ee
This shows that a $\sigma$-finite measure $(\Omega, \ceps, \pr)$
generates a conditional probability space
$(\Omega, \ceps, \cB, \pr)$.
A conditional probability space is a measurable
space equipped with a consistent family of
conditional probabilities indexed by a bunch
\citep[Def.2.2.2]{RENYI}.
The \Renyi structure theorem shows that
every conditional probability space is generated
by a corresponding $\sigma$-finite measure
\citep{RENYI,TaraldsenLindqvist16renyi}.
It should be noted, however,
that the above construction given
by equation~(\ref{eqRenyiCond}) gives a maximal bunch
and then a maximal family of
conditional probabilities.
Consequently,
every conditional probability space can be extended
to a maximal conditional probability space.

It can be noted that the family of conditional probabilities
$\pr (A \st B)$ and the family $\cB$ of elementary conditions
are unchanged if $\pr$ is replaced by $c \pr$ where
$c$ is a positive constant.
The \Renyi state defined by $\pr$ is
the equivalence class $[\pr] = \{c \pr \st c > 0\}$.
The measures $\pr$ and $c \pr$ are equivalent
when interpreted as \Renyi states,
and the conditional probabilities $\pr (A \st B)$ give the
philosophical interpretation in statistical models.  
A \Renyi space is here defined to be a measurable
space equipped with a \Renyi state.
It corresponds to a conditional probability space where the
bunch is maximal.
Our definition here of a \Renyi space is equivalent with
the definition of a full conditional
probability space as used by \citet[p.43]{RENYI}.
We will follow conventional abuse of notation
and use the same symbol $\pr$ for the equivalence class,
a representative $\sigma$-finite measure, and the family of conditional measures. 

Consider now again the commutative diagram in
Figure~\ref{figStatMod} corresponding to a general
statistical inference problem.
It can be interpreted as before
also when $\Omega$ is assumed to be a \Renyi space.
A random quantity $Z$ is a measurable function.
It is said to be $\sigma$-finite if the law
$\pr_Z$ is $\sigma$-finite.
The $\sigma$-finite functions define the natural
arrows in the category of \Renyi spaces.
In the case $\Omega_Z = \RealN$ our definition
of $Z$ being $\sigma$-finite is equivalent with
$Z$ being a regular random variable as defined by
\citet[p.73]{RENYI}.

The prior $\pr_\Theta (A) = \pr (\Theta \in A)$
defines a \Renyi state if $\Theta$ is $\sigma$-finite. 
The interpretation is in terms of the
conditional probabilities $\pr_\Theta (A \st B)$
for $B \in \cB_\Theta = \{B \st 0 < \pr_\Theta (B) < \infty \}$.
If the variable $Z = (X, \Theta)$ and the
data $X$ are $\sigma$-finite,
then the posterior $\pr^x_\Theta$ is well defined with
$\pr_\Theta^x (\Omega_\Theta) = 1$.
This is discussed and exemplified by
\citet{TaraldsenLindqvist10ImproperPriors}
and \citet{LindqvistTaraldsen18proper}.
In the next section this theory will
be generalised so that the posterior $\pr^x_\Theta$
is also allowed to be a conditional \Renyi state as needed for
the butterfly,  Poisson process, and Bernoulli examples in Section~\ref{sIntroduction}

\section{Improper posteriors as conditional \Renyi states}
\label{sTheory}

\citet{TaraldsenLindqvist10ImproperPriors,TaraldsenLindqvist16renyi}
define the posterior law
$\pr^x (A) = \pr (A \st X=x)$ for the case where the
data $X$ is $\sigma$-finite.
The aim now is to prove existence and uniqueness
of a posterior law without assuming that $X$ is $\sigma$-finite.
The simple idea in the following is to define
$\pr^x (A)$ from a family $\pr^x (A \st B)$ indexed
by the elementary conditions $B \in \cB$.
The later is defined by the family of conditional
probabilities $\pr (A \st B)$ defined by the \Renyi state $\pr$.
It is assumed throughout this Section that
$(\Omega, \cE, \pr)$ is a \Renyi space,
and that all random quantities are defined on this space.  
The bunch $\cB$ associated to $\Omega$ is the family
of events $B \subset \Omega$ defined by the requirement
$0 < \pr (B) < \infty$.

Assume that $T$ is a random quantity.
If $B \in \cB$, then
\be{TheCondExpImp}
\pr ((T \in C) A \st B) = \int_C \pr^t (A \st B) \, \pr_T (dt \st B)
\ee
defines $\pr^t (A \st B)$ similarly to how
$\pr^t (A)$ was defined by equation~(\ref{eqTheCondExp}).
The left-hand-side defines for each event $A$
a measure on $\Omega_T$
which is absolutely continuous with respect to $\pr_T (dt \st B)$,
and $h(t) = \pr^t (A \st B)$ is the unique density
$h \in L^1 (\pr_T (dt \st B))$ obtained from
the Radon-Nikodym theorem \citep[p.121]{RUDIN}.

If $X$ is a random quantity, then
the previous defines a family of posterior laws
$\pr^x (\cdot \st B)$ indexed by $B \in \cB$.
This is the necessary ingredient for the interpretation
of a posterior law.
This family is taken as the definition of the posterior law $\pr^x$.
The construction holds also more generally
for a conditional probability space with an arbitrary bunch.
In the following we restrict attention to \Renyi spaces.
The posterior law defines then a conditional \Renyi state.

The next aim is to prove existence of a posterior law $\pr^x$ directly,
and show that
\be{eqMAIN}
\pr^x (A B) = \pr^x (A \st B) \pr^x (B), \;\;\;
\forall A \in \ceps, \forall B \in \cB, \forall' x \in \Omega_X
\ee
This generalization of the structure theorem of
\Renyi is the main result given below in Theorem~\ref{theo1}.
Its precise statement requires some more definitions.

A $\sigma$-finite measure $Q_T$
is by definition a pseudo-law of a random
quantity $T$ if $\{C \st Q_T (C) = 0\} = \{C \st P_T (C) = 0\}$.
If $\tilde{Q}_T$
is another pseudo-law,
then the Radon-Nikodym theorem gives existence
of a unique $c > 0$ in $L^1_{\text{loc}} (Q_T)$ such that 
$\tilde{Q}_T (dt) = c(t) Q_T (dt)$.
Existence of a pseudo-law follows by
defining $Q_T (C) = Q (T \in C)$ where
$Q$ is a probability measure such that
$P (d\omega) = W (\omega) Q (d\omega)$ with $W > 0$
\citep[6.9 Lemma]{RUDIN}.
Given a pseudo-law $Q_T$,
or more generally a law $Q_T$ that dominates $\pr_T$,
we define the conditional law $\pr^t$ by the relation
\be{TheCondExpImp2}
\pr ((T \in C) A) = \int_C \pr^t (A) \, Q_T (dt)
\ee
The left-hand-side defines for each event $A$
a measure on $\Omega_T$
which is absolutely continuous with respect to $Q_T$,
and $h(t) = \pr^t (A)$ is the unique density
obtained from
the Radon-Nikodym theorem \citep[p.121, p.123]{RUDIN}.
If $c(t) > 0$, then the previous shows
that $c (t) \pr^t (d\omega)$ is the conditional law
corresponding to the pseudo-law $Q_T (dt) / c(t)$.
This defines an equivalence between conditional laws,
and defines the unique conditional \Renyi state $\pr^t$
as an equivalence class.
The main mathematical result can now be stated.
\begin{theo}
\label{theo1}  
A random quantity $T$
determines a unique conditional \Renyi state $\pr^t$,
and a unique family of conditional \Renyi states
$\pr^t (\cdot \st B)$ for $B \in \cB$
such that $\forall A \in \ceps$
\be{TheRenyiPseudoLink}
\pr^t (A B) = \pr^t (A \st B) \pr^t (B), \;\;\; \forall' t \in \Omega_T
\ee
\end{theo}
\begin{proof}
All that remains to prove is equation~(\ref{eqTheRenyiPseudoLink}).  
Observe first that
$\pr_T (C \st B) = \pr ((T \in C) B) / \pr (B)$ and
$\pr ((T \in C) B) = \int_C \pr^t (B) Q_T (dt)$ give
$$\pr_T (dt \st B) = \frac{\pr^t (B)}{\pr (B)} \, Q_T (dt)$$
Using this gives
$\int_C \pr^t (A B) Q_T (dt) / \pr (B)$
$=$
$\pr ((T \in C) A \st B)$
$=$
$\int_C \pr^t (A \st B) \pr_T (dt \st B)$\\
$=$
$\int_C \pr^t (A \st B) \pr^t (B) Q_T (dt) / \pr (B)$,
so
$$\int_C \pr^t (A B) \, Q_T (dt) = \int_C \pr^t (A \st B) \pr^t (B) \, Q_T (dt)$$
and equation~(\ref{eqTheRenyiPseudoLink}) is
proved.
\end{proof}
All of the previous can be repeated with a replacement of
the measurable set $A$ with a positive measurable function
$A: \Omega \into [0,\infty]$ and
$\E^t (A) = \pr^t (A)$.
Conditional
expectation of complex valued functions can be defined
by decomposition in positive and negative parts and then in real 
and complex parts.
Consideration of the dual space gives conditional expectation of
a separable Banach space valued $A: \Omega \into \setB$.
The conditional expectation $\E (A \st T=t)$
is in particular well defined when $A$ takes values in a separable Hilbert space.
Separability is assumed to ensure almost everywhere definition
on $\Omega_T$.

\vekk{ 

It does not follow that $\pr^t$
in equation~(\ref{eqTheCondExp}) (Kolmogorov case) or in
equation~(\ref{eqTheCondExpImp2})
(\Renyi case)
is a measure for almost all $t$.
The expectation $\E^t (W) = \int W(\omega) \, \pr^t (d\omega)$,
together with a corresponding change-of-variables formula,
can nonetheless be defined alternatively
similarly to how the ordinary integral is defined.
This was demonstrated already by \citet[p.55, eq.11]{KOLMOGOROV}
in his foundational work,
but this seems to have been mostly forgotten by later writers.
We have, in fact, found no reference to this Kolmogorov integral in the literature.
If there exist a version such that $\pr^t$
is a probability for all $t$,
then $\bmu$ defined by $[\bmu (\omega)](A) = \pr^{T(\omega)} (A)$
is a random probability measure on $\Omega$
\citep[p.1]{Kallenberg17randomMeasure}.
Setting $W = \bmu$ gives then an example of a random quantity $W$ where
$\Omega_W$ is the set of probability measures on $\Omega$.
This existence can, however, not be assumed in general,
and it is not necessary as noted by \citet[p.54-55]{KOLMOGOROV}.
The result is then a weak random measure $\bmu$.

A weak random measure $\bmu$ defined on a fixed $\sigma$-field $\cF$
in a measurable space $\setX$ is
defined by the properties:
(i) $\bmu (\emptyset) = 0$,
(ii) $A \in \cF$ implies $0 \le \bmu (A) \le \infty$,
(iii) If $A_1, A_2, \ldots \in \cF$ are disjoint, then
$\bmu (\sum_i A_i) = \sum_i \bmu (A_i)$.
It is assumed that $\bmu (A)$ is a random quantity
for each measurable set $A$,
and all properties above are interpreted as holding
almost surely with respect to the underlying law
$\pr$ of the \Renyi space $\Omega$.
If $X$ and $T$ are random quantities, then
$\pr^T$, $\pr^T (\cdot \st B)$, and $\pr_X^T$ are all
examples of weak random measures.
The latter is defined, as before, by
$\pr^t_X (A) = \pr^t (X \in A)$ and
$\bmu (\omega) = \pr^{T(\omega)}_X$.
The integral $\bmu (W)$ is defined as by
\citet[p.55, eq.11]{KOLMOGOROV}.
Alternatively as mentioned above,
for the case of a conditional expectation,
it can be defined by $\E^t (W) = \pr^t (W)$ and
equation~(\ref{eqTheCondExpImp2}) with
$A$ replaced by a positive measurable function $W$.

The concept of a weak random measure is here introduced
similarly to how \citet[p.1-2]{Skorohod84} introduces the concept of
a strong random operator.
He also defines the notion of a weak random operator by duality,
but this is equivalent with strong when the image space is the complex numbers.
It should be noted that the naming convention here is
counter intuitive in the sense that a strong random operator is a weaker
concept than a random operator,
but there are good reasons for adopting the conventions of Skorohod.

}

Conditional expectation with respect
to a $\sigma$-field $\cT \subset \ceps$
is defined by $\E (W \st \cT) = \E (W \st T)$ where $T (\omega) = \omega$ and
$(\Omega_T, \cE_T) = (\Omega, \cT)$.
It can be noted that we define
$\pr^t (A)$ directly following Kolmogorov instead of more indirectly
by first defining $\pr (A \st \cT)$ as is more common.
This has the advantage
of allowing a completely general measurable space $\Omega_T$,
whereas the common approach requires separability
properties according to \citet[p.616, Prop.B.24]{SCHERVISH}.

\section{Examples}
\label{sExamples}
\subsection{The uniform \Renyi state on \RealN}

The most familiar example of an improper prior is
given by Lebesgue measure $m$ on the real line $\Omega_\Theta = \RealN$
equipped with the family $\ceps_\Theta$ of Borel sets $A$. 
The corresponding \Renyi state is the equivalence class
$\pr_\Theta = [m] = \{c m \st c > 0\}$.
The \Renyi state is equivalently given by the bunch
$\cB_\Theta = \{B \in \ceps_\Theta \st 0 < m (B) < \infty \}$,
and the family of conditional probabilities
$\pr_\Theta (A \st B) = m (A B) / m (B)$ for $B \in \cB_\Theta$
and $A \in \ceps_\Theta$.
This defines a {\em full} conditional
probability space in the sense of \citet[p.43]{RENYI}, or equivalently, in our terminology, a \Renyi space.

In the context of statistical modeling it is furthermore assumed,
as in Figure~\ref{figStatMod},
that $\Theta$ is a random variable defined on the
underlying \Renyi space $(\Omega, \ceps, \pr)$
with $\pr_\Theta (A \st B) = \pr (\Theta \in A B)/\pr (\Theta \in B)$.
The latter can also be written as
$\pr_\Theta = \pr \circ \Theta^{-1}$ or as
$\pr_\Theta (A) = \pr (\Theta \in A)$ where, as always, 
$(\Theta \in A) = \{\omega \st \Theta (\omega) \in A\}$.
These equations are
interpreted as given for representative
measures in the equivalence classes.

The family  $\cB_0 = \{[-n,n] \st n \in \NatN\}$ does not contain the empty set, it is closed under finite unions,
and there is a sequence $B_j \in \cB_0$ with $B_1 \cup B_2 \cdots = \Omega_\Theta$.
The family
$\cB_0 \subset \cB_\Theta$ is therefor a bunch.
The family of probability measures
defined by $m_n (A) = m (A B_n)/m(B_n)$ for $B_n = [-n,n]$ defines
a conditional probability space $(\RealN, \ceps_\Theta, \cB_0, \{m_n\})$
in the sense of \citet[Definition 2.2.2, p.38]{RENYI}.
The \Renyi structure theorem ensures that this space is generated
by a $\sigma$-finite measure,
or equivalently, 
that this conditional probability space can be extended to a unique \Renyi space.
This \Renyi space is given by the \Renyi state $\pr_\Theta$ described in the previous paragraphs.

In applications the uniform law on the real line is often
described as the limit of the probability measures $m_n$ as $n \rightarrow \infty$.
The previous paragraph identifies the uniform law not as a limit,
but as given by the collection $\{m_n \st n \in \NatN\}$ of probability measures itself.
The uniform \Renyi state $m = \pr_\Theta$ can, however,
also be obtained as 
$\lim_n m_n = m$.
Each $m_n$ and $m$ are interpreted as \Renyi states.
The limit can be defined as in the convergence of 
conditional probability spaces defined by \cite[p.57]{RENYI},
but also in the sense of convergence of \Renyi states given by
equivalence classes \cite[p.5015]{TaraldsenLindqvist16renyi}\citep{BiocheDruilhet16convergence}.

The interpretation of $\pr_\Theta$ comes from the definition of a conditional probability space
as discussed in more detail by \citet[p.34-38]{RENYI}.
Given that $\Theta \in B$ the law is the probability 
distribution $\pr_\Theta (\cdot \st B)$ concentrated on $B$.
The interpretation of all these conditional probabilities 
can be, depending on the situation at hand, in a frequentist sense or in a subjective Bayesian sense.
This generalizes to other unbounded laws including the priors and
posteriors for the butterfly, Poisson process, and Bernoulli examples in the Introduction. 
It is most important since it gives
the needed interpretation of the mathematical 
theory in the context of statistical inference.
The same interpretation is in particular 
used for both the prior and the posterior.
They are on an equal footing,
and this is how uncertainty 
is represented in the statistical model.

\subsection{Conditional \Renyi state densities}

Assume that $(X,\Theta) \sim f(x,\theta) \mu (dx) \nu (d\theta)$
for $\sigma$-finite measures $\mu$ and $\nu$.
It follows that
$(\Theta \st X = x) \sim f(x, \theta) \nu (d\theta)$ by choosing $Q_X = \mu$.
This can be verified directly by the defining
equation~(\ref{eqTheCondExpImp2}).
It follows in particular that this is consistent with
the definition of an improper posterior
as used by \citet[p.1716]{BiocheDruilhet16convergence}.
The previous can also be reformulated simply as
\be{CondDens}
f(\theta \st x) = f(x, \theta)
\ee
There is no need for a normalization constant since two
proportional densities are equivalent when considered as
conditional \Renyi states.
The symbol $f$ is used here, and in the following,
as a generic symbol for a density and also for conditional densities.
The arguments $x \in \Omega_X$ and $\theta \in \Omega_\Theta$
give the interpretation as different functions.

Let $c$ be a function with $c (\theta) > 0$.
In the context here this statement is interpreted
as stating that $c: \Omega_\Theta \rightarrow \RealN$ is measurable and that
$\pr (c(\Theta) \le 0) = 0$.
Similar context dependent interpretations are also used elsewhere, but then without further explanation. 
It follows then that
$(X \st \Theta = \theta) \sim c (\theta) f(x \st \theta) \mu (dx)$,
and so
\be{ModDens}
f(\cdot \st \theta) = c(\theta) f(\cdot \st \theta) 
\ee
when interpreted as a conditional \Renyi densities.
We will then also write
$f(x \st \theta) = c(\theta) f(x \st \theta)$
with this interpretation.
The resulting equivalence class of conditional densities
is the conditional \Renyi state density.
These observations are special cases of the discussion before Theorem~\ref{theo1}
leading to the definition of a conditional \Renyi state as an equivalence class.

A formal prior density $f (\theta)$
gives the joint density $c (\theta) f(x \st \theta) f (\theta)$,
and this shows that the interpretation of
$f (\theta)$ as prior information is dubious in this case.
It is only when the density $f (x \st \theta)$ is normalized
that the common procedure of combining a prior density
$f (\theta)$ with the model density $f (x \st \theta)$ into
a resulting joint density
$f (x, \theta) = f (x \st \theta) f(\theta)$
and a posterior density $f (\theta \st x)$ is well defined.
In all cases, however, the posterior density $f (\theta \st x)$
is well defined as a conditional \Renyi state density
from the joint density $f(x, \theta)$ as in equation~(\ref{eqCondDens}).
In general, the problem with the prior arises when the statistical model $\pr_X^\theta$
itself is allowed to be a conditional \Renyi state.
The likelihood $L (\theta \st x) = f(x \st \theta) = c(\theta) f(x \st \theta)$
is not well defined in this case.

A concrete example with an undefined likelihood is discussed by \citet[p.43]{LavineHodges12icar}
and \citet[p.102]{LindqvistTaraldsen18proper}.
They consider a Gaussian density 
\be{GCAR}
f (x \st \theta) = c(\theta) \exp (- \theta \, \frac{x^T Q x}{2})
\ee
with a known $n\times n$ precision matrix $Q \ge 0$.
This is an improper density if $Q$ has at least one eigenvalue equal to zero,
and then the likelihood is undefined due to the ambiquity introduced by $c$.
The normalization constant $c$ is undefined.
A seemingly natural candidate,
motivated by the proper model case $Q > 0$,
is given by $c(\theta) = \theta^{n/2}$,
and this was used initially in the computer software WinBUGS
\citep[p.102]{LindqvistTaraldsen18proper}.
This choice in WinBUGS was later changed into
$c(\theta) = \theta^{(n-1)/2}$.
It is clear that, in this situation,
a prior information in the form of a prior density $f(\theta)$ can not
be combined with the given improper model to give a well defined
posterior density $f (\theta \st x)$.

A possible solution is given by restricing $x$ to the orthogonal complement
of the null space of $Q$.
The model density is then proper, and $c (\theta) = \theta^{(n-k)/2}$ is the
correct normalization when $k$ is the dimension of the null space of $Q$.
In the $k=1$ case considered by \citet[p.103]{LindqvistTaraldsen18proper} this
corresponds to a change from a uniform to a point mass distribution at $0$
for $(x_1 + \cdots + x_n)/n$.
More generally, the model in equation~(\ref{eqGCAR}) can be further specified as
a Gaussian distribution for $x$ with point masses at $k$ components. 
Anyhow,
a well defined posterior requires that the joint density $f(x,\theta)$,
or more generally as just exemplified,
a well defined joint distribution of the data $X$ and the model $\Theta$.
\citet[p.103]{LindqvistTaraldsen18proper} obtain
a unique normalized posterior $\pr_\Theta^x$
only in the case where the data $X$ is $\sigma$-finite.
Theorem~\ref{theo1} ensures, however, that a unique
posterior \Renyi state is defined also without requiring a $\sigma$-finite $X$.

A more transparent example is given by letting
$\pr_{X,\Theta} (dx,d\theta) = dx d\theta$ correspond to
Lebesgue measure in the plane.
The law of $X$ given $\Theta = \theta$ and the
posterior law of $\Theta$ given $X=x$ correspond then both
to Lebesgue measure on the line.
The factorization $f(x,\theta) = 1 = c(\theta) \pi (\theta)$
with $\pi(\theta) = 1/c(\theta)$ is completely arbitrary.
This can be interpreted according to \citet[p.26]{HARTIGAN}
as saying that the marginal law is not determined
by the joint law.
The choice of a pseudo-law $Q_\Theta$ plays a role similar
to the role of choosing a marginal law in the theory of Hartigan.
The interpretation of Hartigan is discussed in more detail by
\citet{TaraldsenLindqvist10ImproperPriors},
but it differs from the interpretation here.
We insist that the marginal law of $\Theta$ is uniquely determined from
the joint law of $X$ and $\Theta$.
In the case here it is given by the 
measure $\pr_\Theta (d\theta) = \infty \cdot d\theta$
which is not $\sigma$-finite.
It follows in particular that the decomposition
$\pr_{X,\Theta} (dx,d\theta) = \pr_X^\theta (dx) \pr_\Theta (d\theta)$
fails in this case.
However,
regardless of the choice of a pseudo-law $Q_\Theta$,
the decomposition
$\pr_{X,\Theta} (dx,d\theta) = \pr_X^\theta (dx) Q_\Theta (d\theta)$
defines $\pr_X^\theta (dx) = dx$ uniquely as a conditional \Renyi state.

\subsection{Elementary conditional \Renyi states}
Let $\pr_\Theta (d\theta) = d\theta_1 d\theta_2$ be Lebesgue measure in the plane,
and consider the indicator function of the upper half plane:
$\Gamma = \psi (\Theta) = (\Theta_2 > 0)$.
It follows that $\pr_\Gamma = \infty \delta_0 + \infty \delta_1$
so $\Gamma$ is not $\sigma$-finite.
The conditional law
$\pr_\Theta^\gamma (d\theta) = [(\gamma=1)(\theta_2 > 0) + (\gamma=0)(\theta_2 \le 0)] d\theta_1 d\theta_2$ is, however,
a well defined unique conditional \Renyi state.
It corresponds to the dominating measure $Q_\Gamma = \delta_0 + \delta_1$.
The conditional law $\pr_\Theta^1$ is Lebesgue measure restricted to the
upper half-plane and $\pr_\Theta^0$ is Lebesgue measure restricted to
the lower half plane.
This demonstrates directly that
the conditional law is also defined when $\Gamma$ is not $\sigma$-finite.

Consider more generally
a random natural number $T: \Omega \into \NatN$.
A dominating measure for $T$ is the counting measure
$Q_T$ on $\NatN$.
This gives
$\pr (A \st T=t) = \pr^t (A) = \pr (A (T=t))$.
Let $B = (T=1)$.
The previous gives then 
the elementary definition of the law
\be{ElementaryCond}
\pr (A \st B) = \pr (A B), \;\;\; \pr (B) > 0
\ee
The conditional $\pr (A \st B)$
is not defined from this argument when $\pr (B) = 0$
since $\pr^1 (A)$ can be arbitrarily specified in this case.
The previous
is consistent with the familiar
$\pr (A \st B) = \pr(A B)/\pr(B)$
for the case where $0 < \pr (B) < \infty$.
A \Renyi state is arbitrary up to multiplication by a positive constant.
It is an equivalence class of $\sigma$-finite measures.
Theorem~\ref{theo1} gives the existence of conditional expectations
in full generality - including this elementary case.
The restriction $0 < \pr (B) < \infty$ for defining 
$\pr (A \st B)$ has here been relaxed to the condition
$\pr (B) > 0$ by Theorem~\ref{theo1}.

\subsection{The marginalization paradox}

\citet[p.370]{StoneDawid72} consider inference for the ratio $\theta$ of two exponential means.
They assume that $X$ and $Y$ are independent exponentially distributed with
hazard rates $\theta \phi$ and $\phi$ respectively,
so $Z = Y/X$ will have a distribution that only depends on $\theta$.
In fact, $Z = \theta F$, where $F$ has a Fisher distribution with $2$ and $2$
degrees of freedom since a standard exponential variable is distributed
like a $\chi^2_2 /2$ variable.
\citet{StoneDawid72} conclude that the density is
\be{mm1}
f (z \st \theta) = \theta^{-1} (1 + z/\theta)^{-2} = \theta (\theta + z)^{-2}
\ee
and that the posterior density corresponding to a prior density $\pi (\theta)$ is
\be{mm2}
\pi (\theta \st z) \propto \frac{\theta \pi (\theta)}{(\theta + z)^{2}}
\ee
A second argument considers a
joint density for $(x,y,\theta,\phi)$ from a joint prior $\pi (\theta) d\theta d\phi$.
This gives
%
$
\pi (\theta, \phi \st x, y) \propto
\pi (\theta) \theta \phi^2 \exp(-\phi (\theta x + y))
$,
%
and the posterior density of $\theta$ follows by integration over $\phi$ to be
\be{mm4}
\pi (\theta \st x, y)
\propto \frac{\theta \pi (\theta)}{(\theta x + y)^{3}}
\propto \frac{\theta \pi (\theta)}{(\theta + z)^{3}}
\ee
Equation~(\ref{eqmm4}) gives a posterior given the data $(x,y)$
that differs from the posterior found in equation~(\ref{eqmm2}).
This constitutes the argument and paradox presented originally by
\citet{StoneDawid72}.

\vekk{ 
A range of similar paradoxes were presented later by
\citet{DawidStoneZidek73}. 
Furthermore, Lindley,
in his discussion of the paper
\citep[p.218]{DawidStoneZidek73} writes:
\begin{quote}
{\it The paradoxes displayed
here are too serious to be ignored and impropriety must go.
Let me personally retract
the ideas contained in my own book.
}
\end{quote}  
This is of particular relevance here since in 1964,
in his book,
\citet[p.xi]{LINDLEY} wrote:
\begin{quote}
{\it
  The axiomatic structure used here  is  not the usual one
  associated with the name of Kolmogorov.
  Instead one based on the ideas  of  Rényi has been used.
}
\end{quote}  
We argue here and in the following that Lindley's initial intuition
was correct:
The theory of Rényi gives a mathematical foundation
for statistics that allows unbounded measures.
} 


We will next reconsider the above example
in view of the theory presented in the previous Section.
This has already been indicated by \citet{TaraldsenLindqvist10ImproperPriors},
and is discussed in more detail by \citet{LindqvistTaraldsen18proper}.
\citet{LindqvistTaraldsen18proper} rely on a theory where
it is only allowed to condition on $\sigma$-finite statistics.
We extend this argument now with reference to Theorem~\ref{theo1}
which allows conditioning on any statistic.

The initial assumptions are equivalent with a
joint distribution given by the density:
\be{mm5}
f (x,z,\theta,\phi) = \pi (\theta) f(x,z \st \theta, \phi)
= \pi(\theta) \theta \phi^2 x e^{-\phi x (\theta + z)} 
\ee
%
Integration over $\phi$ gives
\be{mm6}
f (x,z,\theta) = \pi(\theta) \theta x^{-2} (\theta + z)^{-3}
\ee
which implies
\be{mm7}
\pi (\theta \st x,z) =
\pi(\theta) \theta x^{-2} (\theta + z)^{-3} =
\pi(\theta) \theta (\theta + z)^{-3}
\ee
The second equality holds since it is equality in the sense
given by an equivalence class as in Theorem~\ref{theo1}.
The right hand side can be multiplied by an arbitrary positive function $c(x,z)$
without changing the equality sign.
Equation~(\ref{eqmm7}) is equivalent with equation~(\ref{eqmm4}) since
there is a one-one correspondence between
$(x,y)$ and $(x,z)$.

An alternative is to integrate equation~(\ref{eqmm5}) over $x$ to obtain
\be{mm8}
f (z,\theta, \phi) = \pi(\theta) \theta (\theta + z)^{-2}
\ee
which implies
\be{mm9}
\pi (\theta \st z, \phi) = \pi(\theta) \theta (\theta + z)^{-2}
\ee
This is similar to equation~(\ref{eqmm2}),
but the conditioning differs.

Reconsider now the argument leading to equation~(\ref{eqmm2}).
The first observation was that $Z = Y/X$ has a distribution
that only depends on $\theta$.
This is true, but it is still conditionally given both $\theta$ and $\phi$
as assumed initially in the model.
Equation~(\ref{eqmm2}) and equation~(\ref{eqmm1}) are wrong as stated,
interpreted as conditional \Renyi states,
but can be corrected by a replacement of
$f (z \st \theta)$ by $f (z \st \theta, \phi)$ and
$\pi (\theta \st z)$ by $\pi (\theta \st z, \phi)$.
The error in the original argument,
as interpreted in the theory presented here,
is that it can not be concluded that
$\pi (\theta \st z) = \pi(\theta \st z, \phi)$
even though the later does not depend on $\phi$.
Equation~(\ref{eqmm9})
is not in conflict with
equation~(\ref{eqmm7}) for the same reason.

\vekk{ 
Starting with either equation~(\ref{eqmm6}) or
equation~(\ref{eqmm8}) gives
\be{mm10}
f (z,\theta) = \infty \cdot \pi(\theta)
\ee
which shows that neither
$Z \st \Theta=\theta$
nor
$\Theta \st Z=z$
can be represented by a $\sigma$-finite measure.
This implies that the argument in equations~(\ref{eqmm1}-\ref{eqmm2})
is wrong given equation~(\ref{eqmm5}).
Equation~(\ref{eqmm4}), or equivalently
equation~(\ref{eqmm7}), gives the correct posterior distribution for $\Theta$.

If, instead, the prior $\pi(\theta) \phi^{-1} d\theta d\phi$ is used,
then the result will be
\be{mm11}
f (\theta \st z, \phi) = \phi^{-1} \pi(\theta) \theta (\theta + z)^{-2}
= \pi(\theta) \theta (\theta + z)^{-2}
\ee
and
\be{mm12}
f (\theta \st x,z) = \pi(\theta) \theta x (x (\theta + z))^{-2}
= \pi(\theta) \theta (\theta + z)^{-2}
\ee
The conditionals coincide,
but it is still true that neither equals the law of $\Theta \st Z=z$
since the law of $(\Theta, Z)$ still fails to be $\sigma$-finite.

If, however, equation~(\ref{eqmm1}) together with a prior
$\pi (\theta)$ is taken as the initial $\sigma$-finite law for $(\Theta,Z)$,
then equation~(\ref{eqmm2}) is the correct posterior.
There is then no paradox since the conflicting conclusions
are consequences of different initial assumptions.
} 

More generally,
it can be noted that
even if a conditional law $\pr^{x,z}$ does not depend on $x$
it can not be concluded that it equals $\pr^z$.
This is demonstrated by equation~(\ref{eqmm7}) and
equation~(\ref{eqmm9}).
The rule $\pr^{x,z} = \pr^z$ holds for probability distributions,
and also more generally if $Z$ and $(X,Z)$ are $\sigma$-finite
given that $\pr^{x,z}$ does not depend on $x$.
\citet{StoneDawid72} calculated formally 
as if the rule where generally valid.
This resulted in two conflicting results.
This example, and the other examples constructed by \citet{StoneDawid72} are most important
since they illustrate important differences between the
theories of Kolmogorov and \Renyi.
\citet{StoneDawid72} pointed out that purely formal manipulations with improper distributions,
treated as if they obeyed all the rules of proper distributions,
could lead to paradoxical inconsistencies
— which by reductio ad absurdum —
is an argument against doing such formal computations.

\subsection{The Jeffreys-Lindley paradox}

Observations can give rejection of a simple
hypothesis at the $5\%$ level,
but a Bayesian analysis can give the hypothesis a posterior probability
larger than $95\%$.
\citet{Lindley57paradox} discussed this seemingly paradoxical phenomena 
with reference to previous work by
\citet{JEFFREYS}.
Both \citet[p.148-156]{BERGER} and \citet[p.230-236]{ROBERT} give
thorough discussions of the problem of testing a point null hypothesis,
and explain that the use of improper priors is a delicate issue in this case.
This has also been emphasized in several discussion papers
\citep{Shafer82paradoxTesting,BergerSellke87pointNull,BergerDelampady87pointNull,
RobertChopinRousseau09jeffreys,Robert14jlParadox}.
A full discussion of this problem in the context of the theory of \Renyi will not be given here,
but we will indicate some consequences and observations.

The most important is to note that any prior,
improper or not, contains information.
We agree with \citet[p.29]{ROBERT} that
it is a mistake to think of improper priors as representing ignorance. 
This is particularly important when testing a point null hypothesis,
which in most situations implies a non-symmetric treatment of the
hypothesis and the alternative hypothesis.
The relevance of the information is specific to each particular case with its
own interpretation.
\Renyi explained that improper laws
can be interpreted in terms of the associated family of conditional probabilities.
This holds for both prior and posterior laws,
and also so in a hypothesis testing problem.

Assume that $x \sim \normvar(\theta,\sigma^2)$ with
unknown mean $\theta \in \Omega_\Theta = \RealN$ and known variance $\sigma^2$
so
$f(x \st \theta) = [\sqrt{2\pi} \sigma]^{-1} \exp({-\half (x - \theta)^2/\sigma^2})$.
Consider the hypothesis $H_0 = \{0\} \subset \Omega_\Theta$ versus
the alternative $H_1 = H_0^c = \{\theta \in \Omega_\Theta \st \theta \not\in H_0\}$.
This basic hypothesis testing problem is often the first example
of hypothesis testing presented to statistics students using
the notation $H_0: \theta = 0$ versus $H_1: \theta \neq 0$.
Our notation identifies the hypothesis and its alternative
more explicitly with a partition $H_0 + H_1$ of the
model parameter space $\Omega_\Theta$.

The uniformly most powerful unbiased level $\alpha$ test rejects $H_0$ if
\citep[p.374]{CasellaBerger90inference}
\be{pValue}
t = \phi (x) 
=
2 \Phi (-\abs{x/\sigma}) \le \alpha
\ee
where $\Phi (z) = \pr (Z \le z) $ with $Z \sim \normvar (0,1)$.
The test statistic $T = \phi (X)$ is the p-value.
It is a probability,
but it must not be confused with the posterior probability of $H_0$ given the data.
The posterior probability is undetermined in this classical analysis. 

Consider next a Bayesian analysis with
a prior density $\pi$ with respect to the measure
$\nu (d\theta) = \delta_0 (d\theta) + d\theta$.
The Dirac measure $\delta_0$ is dimensionless,
and it is hence assumed that $\theta$, $x$, and $\sigma$ are dimensionless in the following.
A Bayesian test with minimal posterior risk rejects $H_0$ if
the posterior probability of $H_0$ given the data is small
\citep[p.164]{BERGER}
\be{BReject}
s = \phi_\pi (x) =
\pi (0 \st x) =
\left[1 +
  \int
  \frac{f(x \st \theta) \pi (\theta) d\theta }{f(x \st 0) \pi (0)} \right]^{-1}
\le
\frac{\LossII}{\LossII + \LossI}
\ee
$\LossI$ is the loss corresponding to a type I error,
$\LossII$ is the loss corresponding to a type II error,
and the loss is zero otherwise.
The classical and Bayesian tests are similar in form,
but the Bayesian test statistic $S = \phi_\pi (X)$ depends on the prior
density $\pi$.
We have here restricted attention to the case where the posterior is proper,
and the above integral is then finite. 

Consider first the constant prior $\pi_\infty (\theta) = c$.
This gives
\be{p1}
\phi_\infty (x) 
= [1 + f(x \st 0)^{-1}]^{-1}
= [1 + \sqrt{2 \pi} \sigma \exp \left(\frac{x^2}{2 \sigma^2}\right)]^{-1} 
\ee
Figure~\ref{fig1}
\begin{figure}
    \centering
    \includegraphics[width=.8\textwidth]{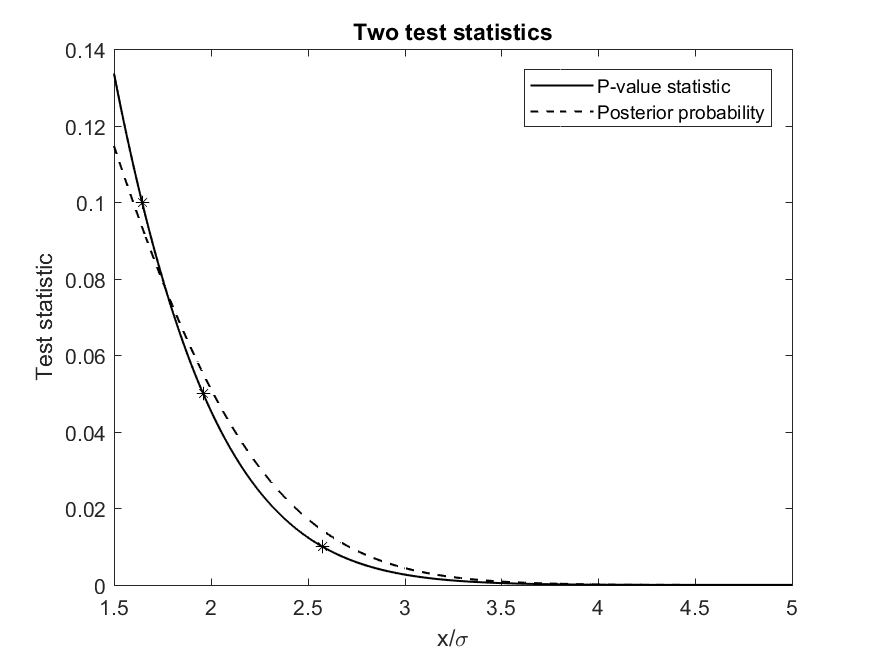}
    \caption{Two test statistics for testing $H_0: \theta = 0$ versus $H_1:\theta \neq 0$
      based on $x \sim \normvar(\theta, \sigma^2)$.} 
    \label{fig1}
\end{figure}
shows the remarkable similarity between the
p-value and the posterior probability for the case $\sigma = 1$.
\citet[p.234]{ROBERT} also notes this similarity in his Table~5.2.5,
and discusses this phenomena.
This similarity,
and more generally the close resemblance in practice between Bayesian and classical methods
for many common statistical problems,
is a common theme in the early fundamental texts on Bayesian statistics
\citep{JEFFREYS,LINDLEY,SAVAGE}.

A common way to justify usage of improper priors is to
consider limits of proper priors.
Consider the sequence $\pi_1, \pi_2, \ldots$ of proper prior densities defined by
$\pi_n (0) = \pi (0)$ and
$\pi_n (\theta) = (1 - \pi (0)) g_n (\theta)$ for $\theta \neq 0$.
Let $g_n (0) = 0$ and
$g_n (\theta) = [\sqrt{2\pi} \tau_n]^{-1} \exp(-\half \theta ^2/\tau_n^2)$
for $\theta \neq 0$.
Define also $g_\infty (0) = 0$ and $g_\infty (\theta) = 1$ for $\theta \neq 0$.
If the variance $\tau_n^2 \rightarrow \infty$,
then $g_n \rightarrow g_\infty$ intuitively as densities since
a normally distributed variable with infinite variance
should correspond to a variable with a constant density.
This convergence is in fact true if interpreted as densities with respect to $d\theta$ 
in the sense of $q$-vague convergence \citep{BiocheDruilhet16convergence}. 
It seems hence reasonable to take the sequence
$\pi_1, \pi_2, \ldots$ of proper densities as an approximation
of the improper density $\pi_\infty (\theta) = c = \pi (0)$.
The point mass $\pi (0)$ at $\theta=0$ is then fixed,
and the densities for $\theta \neq 0$ approximate the flat density.
Equation~(\ref{eqBReject}) gives, however, that
$\pi_n (0 \st x) \rightarrow 1$,
since  $\int f(x \st \theta) \pi_n (\theta) \,d\theta \rightarrow 0$.
This is in conflict with $\pi_\infty (0 \st x) \le [1 + \sqrt{2 \pi} \sigma]^{-1}$.
The source of the problem is that the generalized density
$\pi_n = \pi (0) \delta_0 + (1 - \pi (0)) g_n$ with respect to $d\theta$ does not converge to the density
$\pi_\infty = \pi (0) \delta_0 + (1 - \pi (0)) g_\infty$ as intuition would suggest,
but instead $\pi_n \rightarrow \delta_0$ as explained by
\citet{BiocheDruilhet16convergence}.

The previous can be used to illustrate the Jeffreys-Lindley paradox.
Assume that $x$ is statistically significantly different from $0$ at the $\alpha$ level
of significance in the sense of $\phi (x) \le \alpha$ as defined by equation~(\ref{eqpValue}).
The convergence $\pi_n (0 \st x) \rightarrow 1$ ensures, however,
that the posterior probability $\pi_N (0 \st x) > 1 - \alpha$ for a prior $\pi_N$ with a
sufficiently large $N$.
This can, of course, only be considered to be paradoxical in a situation where
the prior $\pi_N$ is reasonable.
Consider instead a symmetric proper prior on the form
$\pi (0) = 1/2$ and $\pi(\theta) = g(\theta)/2$ for $\theta \neq 0$ where
$g(\abs{\theta})$ is non-increasing in $\abs{\theta}$.
The critical values $1.645$, $1.960$ and $2.576$ for $x/\sigma$
corresponds to the p-values $10\%$, $5\%$ and $1\%$ as also indicated
in Figure~\ref{fig1}.
The corresponding posterior values $\pi (0 \st x)$ are, however,
bounded from below by $39\%$, $29\%$, and $10\%$ for {\em any}
prior on the given form \citep[p.154, Table 4.4]{BERGER}.
The conclusion is that a large class of reasonable symmetric proper priors
gives a posterior probability much larger than the classical $p$-value.
This is an even more striking illustration of the
Jeffreys-Lindley paradox.

Consider again the improper prior density $\pi_\infty (\theta) = c$
with respect to $\nu (d\theta) = \delta_0 (d\theta) + d\theta$.
The value of the constant $c$ is of no concern, as we know from the general theory,
and also explicitly from equation~(\ref{eqp1}).
Equation~(\ref{eqp1}) has, however, a dependency on $\sigma$ that is a concern.
It follows that the prior $\pi_\infty$ corresponds effectively to two
different priors if a concrete problem is formulated first in terms of one measurement scale and then
alternatively in terms of a different measurement scale.
The p-value does not share this defect, since it only depends
on the scaled variable $x^* = x/\sigma$ as given in
equation~(\ref{eqpValue}).
An alternative prior is obtained
by reformulating the original problem in scaled variables and
using the prior $\pi_\infty$ for this problem.
Transforming back gives the
improper density $\pi^* (0) = \sigma$ and
$\pi^* (\theta) = 1$ for $\theta \neq 0$.
The result is a posterior probability $\pi^* (0 \st x)$ that only depends on $x^* = x/\sigma$,
and it is remarkably close to the p-value as shown in Figure~\ref{fig1}
for all values of $\sigma$.

It was seen above that a seemingly reasonable approximation by proper priors failed.
An alternative sequence of proper priors is obtained by using the
interpretation \Renyi gives for the prior corresponding to the density $\pi^*$.
Let $B_n = [-n \sigma, n \sigma]$ for $n = 1, 2, \ldots$.
The conditional probability $\pr_\Theta (A \st B_n)$
is then given by the proper density $\pi^*_n$ defined by
$\pi^*_n (0) = 1/(1 + 2 n)$ and
$\pi^*_n (\theta) = [-n \le \theta/\sigma \le n]/(\sigma + 2 n \sigma)$ for $\theta \neq 0$.
In this case $\pi_n^* (0 \st x) \rightarrow \pi^* (0 \st x)$,
and this gives in particular a proper prior with a posterior that
approximates the p-value as shown in Figure~\ref{fig1}.
The appropriateness of a prior on this form can not be decided in general,
but must be decided in each concrete case.

Consider finally a concrete problem where it is assumed that the prior density $\pi^*$
gives a reasonable prior for $\theta$.
Assume that a measurement is done and $x \sim \normvar (\theta, \sigma^2)$ is observed.
In this case, the classical and Bayesian procedures are very similar if 
$\alpha = \LossII/(\LossI + \LossII) = 5\%$.
Assume next that the experimenter chooses
to repeat the measurement $N-1$ more times.
The prior information is, of course, not changed by this decision,
so the prior is still given by $\pi^*$.
A sufficient statistic is given by the empirical mean
$\overline{x} \sim \normvar(\theta, \sigma^2/N)$.
The classical p-value and the posterior probability $\pi^* (0 | \overline{x})$
are in this case very different if $N$ s large.
It follows in particular that
$\pi^* (0 | \overline{x}) \rightarrow 1$ as
$N \rightarrow \infty$ for all fixed $\overline{x}$,
and the Jeffreys-Lindley paradox reappears.
We see this, in fact, as no paradox, but as a most important and
striking demonstration of an important
difference between Bayesian and classical inference.

\subsection{Hypothesis testing with improper posteriors}

The possibility of improper posteriors was not considered in the previous discussion of
the Jeffreys-Lindley paradox.
It was, in fact, demonstrated that there exist a proper prior so that
the classical and the Bayesian decision rules essentially coincides
as shown in Figure~\ref{fig1}.
This proper prior appears naturally from the \Renyi interpretation of
a corresponding improper prior in terms of a family of conditional probabilities.
It was also noted that this improper prior can be approximated
arbitrary well by a sequence of proper priors in a natural topology for \Renyi states
given by q-vague convergence \citep{BiocheDruilhet16convergence}.

Another observation is that, in general,
a classical matching prior is typically improper.
\citet{DeGroot73testing} demonstrates, by an elegant argument,
how a matching prior can be determined for a different problem.
A classical matching prior is here defined to be a prior such that the
posterior coincides with the p-value.
The prior $\pi^*$ is only approximately matching as shown in Figure~\ref{fig1}.
A matching prior - if it exists - is determined by the integral equation
that follows by equating $\alpha$ and $\LossII/(\LossI + \LossII)$
in equation~(\ref{eqpValue}) and equation~(\ref{eqBReject}).
We will not discuss this further here,
but observe that a solution is given explicitly by an inverse Fourier transformation.  

The butterfly, Poisson process, and Bernoulli examples in the Introduction
can be used to exemplify a hypothesis testing problem with an improper posterior.
Consider, instead,
testing of $H_0: \gamma \le 0$ versus $H_1: \gamma > 0$
based on observing $x \sim \normvar (\gamma, \sigma^2)$ with
unknown $\theta = (\gamma, \sigma) \in \Omega_\Theta = \RealN \times \RealN_+$.
This problem,
but with known variance $\sigma^2$,
is considered by \citet[p.147-148]{BERGER}.
He notes that a constant prior corresponds to an infinite mass
to both hypothesis,
but argues that this can be tackled by consideration
of increasingly larger intervals.
The essence of the following argument is that
this argument should in principle then be equally possible for the posterior.
This is given by the general interpretation
of any \Renyi state by its corresponding family of conditional probabilities.

Assume that the prior density is $\pi (\theta) = 1/\sigma$
with respect to $\nu (d\theta) = d\gamma d\sigma$.
The posterior is then improper,
and given by the density
$\pi (\theta \st x) = \sigma^{-2} \exp({-\half (\theta - x)^2/\sigma^2})$.
It follows that
$\pr (\Theta \in H_0 \st X=x) = \pr (\Theta \in H_1 \st X=x) = \infty$,
and it is not obvious how to formulate a decision rule.
The interpretation of \Renyi leads to consideration of the posterior probability
$\pr^x_\Theta (H_0 \st B)$ for all $B$ with $0 < \pr (\Theta \in B) < \infty$.
In an application it can, possibly, be argued that
it is sufficient to consider elementary events on the form
$B (m,n) = (-m,m) \times (1/n, n)$.
This determines corresponding posterior
probabilities $\phi_{m,n} (x)$ for $H_0$ which
should be considered when deciding to reject $H_0$ or not.
We leave the further discussion of this for the future. 

Another possible approach is presented next. 
The posterior gives the two improper marginal densities
\be{pSigma}
\pi (\sigma \st x) = \sigma^{-1}
\ee
\vspace{-4ex}
\be{pGamma}
\pi (\gamma \st x) = \abs{\gamma - x}^{-1}
\ee
The posterior for $\sigma$ has no dependence on $x$ as
intuition would suggest without any further argument.
The posterior for $\gamma$ is symmetric around $x$
- again in harmony with intuition.
The posterior for $\gamma$ is improper,
but it clearly represents an 
updating of the state of knowledge regarding $\gamma$.
As a technical aside it can be observed that the posterior for $\gamma$
is in fact not $\sigma$-finite,
but the posterior for $\sigma$ is $\sigma$-finite.
It is only the full posterior that is guarantied to be
represented by a conditional \Renyi state
by Theorem~\ref{theo1}. 

For the given hypothesis testing problem it can seem natural to replace
the model space $\Omega_\Theta$ with the focus space $\Omega_\Gamma = \RealN$. 
The hypothesis are then represented by
$H^*_0 = \{\gamma \st \gamma \le 0\}$ and
$H^*_1 = \{\gamma \st \gamma > 0\}$.
Motivated by the interpretation from \Renyi
it seems natural to consider 
the elementary events
$B (m,n) = (-m,m) \setminus (x-1/n, x+1/n)$
with $m > \abs{x}$.
The singularity of the posterior at $\gamma = x$ can then be
tackled by taking the limit $n \rightarrow \infty$
for the resulting posterior probabilities $\phi_{m,n} (x)$ for $H_0$.
This leads to $\phi_m (x) = \phi_{m,\infty} (x) = (x \le 0)$.
The conclusion from this argument is
to reject $H_0: \gamma \le 0$ if $x > 0$.

\vekk{
\subsection{Toy example}
Improper target density 
\begin{equation}
f(x) = \frac1{\sqrt{2\pi}}e^{-x^2/2}+\frac{f_\infty}{1+e^{-x}}
\end{equation}
This has an improper cdf
\begin{equation}
F(x) = \phi(x) + f_\infty \ln(e^x+1).
\end{equation}
Letting $a_i$, $i=0,2,\dots,m$ denote the boundaries between the intervals 
where $a_0=-\infty$, the theoretically optimal kernel weights for each interval are then
\begin{equation}
w_i = F(a_i)-F(a_{i-1})
\end{equation}

\subsection{Butterflies}
\citep{Tufto2012butterflies}
\subsection{Random effects}

\subsection{Removal sampling}
\citep{Druilhet2016}
}

\section{Final remarks}
\label{sDiscussion}

\citet[p.xi]{LINDLEY} wrote in 1964 in the preface of
his classic book on Bayesian statistics:
\begin{quote}
{\it\bf
  The axiomatic structure used here  is  not the usual one
  associated with the name of Kolmogorov.
  Instead one based on the ideas  of  Rényi has been used.
}
\end{quote}  
It can be concluded that Lindley
initially supported the use of conditional probability spaces
as introduced by Rényi.
We have argued that Lindley's initial intuition is correct.
The theory of \Renyi gives a natural approach
to Bayesian statistics including commonly used objective priors.

The marginalization paradoxes seem to have been the main reason
for Lindley's change in opinion on this.
Tony O'Hagan \href{https://www.youtube.com/watch?v=cgclGi8yEu4}{interviewed}
Lindley for the Royal Statistical Society's Bayes 250
Conference held in June 2013.
Lindley explains very nicely that all probabilities are conditional probabilities,
but also recalls his reaction to the marginalization paradoxes presented
by \citet{StoneDawid72}:
%
{\it Good heavens, the world is collapsing about me.} 
%
In the interview,
Lindley continuous to argue that Bayesian statistics
is a sound theory, 
and that the focus should be on
how to quantify the prior uncertainty of the
unknown parameters.
The parameters should be viewed as real physical quantities 
regardless of which experiment is later used for decreasing their uncertainty.
This clearly disqualifies the choice of data dependent priors,
and even the choice of priors depending on the particular statistical model used.
We wholeheartedly agree with Lindley on this,
but we claim that this can be done also within the more general
theory introduced by \Renyi and continued here.

\vekk{

As mentioned in Section~\ref{sMathStat} we consider choice
of notation to be important, and we next discuss our choices briefly.
The space $\Omega$
is abstract in the sense that it is assumed to exist,
and the definition of other concepts such as expectation value and random variables
relies on its existence.
  \citet[p.607]{SCHERVISH} use the symbol ${\cal S}$
  for the underlying probability space,
  and this is also quite common in more
  elementary texts on statistics \citep[p.6]{CasellaBerger90inference}.
  \citet[p.29]{LehmannRomano05testing} use the symbol ${\cal Z}$,
  \citet[p.187]{HALMOS} use $X$,
  and \citet[p.2]{KOLMOGOROV} used the symbol $E$.
  We use $\Omega$ which is the traditional choice in probability theory
  \citep{DOOB,RENYI,BILLINGSLEY,LOEVE,BREIMAN,Kallenberg97probability}.
  This choice is also used in some reference texts on statistics:
  \citep[p.81]{RAO} and \citet[p.287]{StuartOrd96kendallsvolI}
  use the symbol $\Omega$ for the set of elementary events.

The notation $(Z \in A) = Z^{-1} (A) = \{\omega \st Z (\omega) \in A\}$
is as used by \citet[p.20]{BREIMAN}.
A similar convention,
with obvious generalisations for cases with more conditions,
is used by \citet[p.~4]{DOOB}.
Hence, we prefer $(X \in A, Y \in B)$ in favour of the
alternatives $(X, Y)^{-1} (A, B)$ and $(X^{-1} (A)) \cap (Y^{-1}(B))$. 
Additionally, we can then define
$\pr_X (A) = \pr (X \in A)$, which is familiar in probability theory.
\citet{SCHERVISH} uses $\mu_X$ similarly when the underlying law is $\mu$. 

The notation $\pr^t (A) = \pr (A \st T = t)$ is convenient,
and $\pr_X^\theta (A) = \pr (X \in A \st \Theta = \theta)$ is then a consequence.
Kolmogorov used the same convenience, but with reversed roles for
super- and sub-scripts.
Our subscript choice is motivated by the widespread usage
of $f_X$, $F_X$, $M_X$ etc for the density, cumulative density,
moment generating function and so on for quantities associated with $X$.
The notation $\Omega_X$ and $\Omega_\Theta$ for the
data space and model space respectively are then also natural choices.

Usually, the conditional probability $\pr_X^\theta$ is specified such that
it is a probability distribution on the data space for each $\theta$.
Often this gives a one-one correspondence between $\Omega_\Theta$
and a family of probability distributions on the data space.
In this case it can be assumed directly that the model space
$\Omega_\Theta$ is a set of possible probability distributions on the data space.
There are, however, important exceptions to this case:
Over-parametrised models in linear regression, generalised linear models,
neural networks, and many other related complex  models.

It is quite common to let $P_\theta$ denote the probability law on the sample space
\citep[p.3]{LehmannRomano05testing}\citep[p.83]{SCHERVISH}.
We choose instead to use the notation $\pr_X^\theta$ for this as explained.
\citet[p.83-84]{SCHERVISH} use the notation $\mu_\Theta$ for the prior,
$\mu_X$ for the marginal, and $\mu_{\Theta \st X} (A \st x)$ for the posterior.
Our notation for the later is $\pr_\Theta (A \st X = x)$,
and more generally we prefer to keep the single symbol $\pr$.

The model space can be finite dimensional:
One interesting example is that it equals the unit circle $S^1$.
It can be infinite dimensional as in nonparametric statistics or
as in problems in econometrics with stochastic processes.
The model space comes typically with additional mathematical structure
(differential manifold, group actions, $\ldots$) specific to the
domain of interest.
The same comment holds for the data space, the action space, and the focus space.
Both $\phi_1 (x) = \pr^x_\Theta$
and $\phi_2 (x) = \pr^x_\Gamma$ exemplify actions where the
action space is the set of laws on the model space and the focus space respectively.
In an estimation problem it is natural to choose
the action space equal to the focus space,
but there are many  other possibilities depending on the kind of inference:
hypothesis test,
interval estimate (credibility, confidence, prediction, tolerance, $\ldots$),
distribution estimate (credibility, confidence, $\ldots$),
and many more.

} 

\vekk{ Kutter mer
The concept of 
a conditional probability space invented by Rényi is not commonly used in statistics today.
The example with the uniform law $\pr_\Theta$ on the real line provides
the family $\cB_0 = \{[-n,n) \st n \in \NatN\}$ of events 
which is a {\em bunch} $\cB$ and
the family 
$\{\pr_\Theta (\cdot \st B) \st B \in \cB\}$ 
of conditional probabilities
defining together a conditional probability space:
A measurable space equipped with a consistent family of
conditional probabilities indexed by the sets $B$ in the bunch.
The structure theorem of \citet{RENYI} ensures that any conditional probability space
is generated by a corresponding $\sigma$-finite measure.
We have hence abused notation slightly by using the symbol
$\pr_\Theta$ for both a $\sigma$-finite generating measure
and the corresponding \Renyi state given as an equivalence class.
The latter corresponds to the maximal family of conditional probabilities
indexed by the bunch $\cB_\Theta = \{B \st 0 < \pr_\Theta (B) < \infty\}$.
The model space $\Omega_\Theta$ is a \Renyi space when equipped with
the \Renyi state $\pr_\Theta$.

The previous definitions of a \Renyi space
and a \Renyi state are introduced here,
and are modifications of similar definitions suggested by
\citep[p.5013]{TaraldsenLindqvist16renyi}.
It corresponds to the definition of a
full conditional probability space as introduced by
\citet[Def.2.2.3]{RENYI}.
The bunch for a given conditional probability space
is contained in the bunch defined as above from
a generating $\sigma$-finite measure.
Convergence of \Renyi states given by equivalence classes $[\mu]$ of
Radon measures is investigated by
\citet{BiocheDruilhet16convergence}.
Their results demonstrate also that
the distinction between a $\sigma$-finite measure
and the corresponding \Renyi state is essential.
}

Historically, the most influential initial work on Bayesian inference is possibly
given by the book by \citet{JEFFREYS}.
\citet[p.21]{JEFFREYS} argues in particular that
the normalization of probabilities is a rule generally adopted,
but that the value $\infty$ is needed in certain cases.
This is in line with current usage of Bayesian arguments
\citep{BERGER,ROBERT}.
It is well established that inference based on
the posterior gives, indeed,
a most rewarding path
for obtaining useful inference procedures
from both a frequentist and a Bayesian perspective
\citep{BERGER,SCHERVISH,LehmannRomano05testing}.
Parts of Jeffreys arguments were mainly intuitive, 
and there is
a lack of mathematical rigor as also observed by \citet{RobertChopinRousseau09jeffreys}.
We suggest that 
a rigorous reformulation of 
some of the original and most important 
ideas of \citet{JEFFREYS}
can be done within the mathematical theory presented here.

Within this framework we reach the view that improper posteriors,
just as improper priors, are not `improper' but reflect 
the updated state of knowledge 
about a parameter after conditioning on the data.  
Returning to the introductory Poisson-process
example, at time $t$, we have clearly learned something about $\lambda$ in that
our belief in large values of the Poisson intensity $\lambda$ has decreased while our
relative degree of belief in small values of $\lambda$ has remained approximately unchanged.
An improper posterior does not imply that our prior was wrong,
but only that more data perhaps needs to be collected if possible.
Proceeding by using the improper posterior at time $t$
as prior in subsequent inference, say based on the number of occurrences  observed
in a sufficiently long subsequent interval $(t,t_2]$, we indeed eventually reach
the same proper final posterior as the one reached by combining the initial scale
prior and the likelihood for the data on $(0,t_2]$.
We hope that the reader can appreciate that this argument 
indicates also the potential philosophical importance of unbounded laws more generally.




An unbounded law can, according to \Renyi, be interpreted by
the corresponding family of conditional probabilities given
by conditioning on the events in the bunch.
These elementary conditional probabilities are probabilities in the sense of Kolmogorov,
and the interpretation depends on the application.
They can, as
\citet{Lindley06Uncertainty}
advocates convincingly, be interpreted as personal probabilities
corresponding to a range of real life events.
They can also, as needed in for instance quantum physics,
be interpreted as objectively true probabilities
representing a law for how a system behaves when observed repeatedly under idealized conditions.

Assume now that you accept
a theory where the prior uncertainty is given by a possibly unbounded law.
It is then natural,
we claim, 
that you accept that a resulting posterior uncertainty can also be
represented by a possibly unbounded law.
Both the prior and the posterior represent uncertainty of the same kind.
Hopefully, many can agree on this on an intuitive level.
The main mathematical result presented here is Theorem~\ref{theo1} which provides a
key ingredient in a mathematical model for statistics
in which this can be done consistently without paradoxical results.
This key ingredient is a well defined extension of the
concept of conditional expectation as introduced originally by
\citet{KOLMOGOROV} to also include the case of \Renyi {\ } spaces.


\vekk{

\appendix

\section{Appendix on measure theory}
\label{sDef}

\subsection{Measurable space and measure}

A measurable space is a set $\setX$ equipped with
a $\sigma$-field $\cF$ of subsets of $\setX$.
A $\sigma$-field $\cF$ is a collection of
subsets of a fixed set that contains the empty set $\emptyset$
and is closed under complements and countable unions.
A set $A \subset \setX$ is measurable if $A \in \cF$.
A measure $\mu$ is a function
$\mu: \cF \into [0,\infty]$ with $\mu (\emptyset) = 0$
that is countably additive:
$\mu (A_1 + A_2 + \cdots) = \mu (A_1) + \mu (A_2) + \cdots$.
A probability measure is a measure $\mu$ on
a measurable space $\setX$ with $\mu (\setX) = 1$.
A measure space is a measurable space $\setX$ equiped with a
measure \citep[p.16]{RUDIN}. 
A probability space is a measurable space $\setX$ equipped with a
probability measure.


A sigma-field $\cF_0 \subset \cF$ of a measure
space $(\setX, \cF, \mu)$ is sigma-finite
if there exist measurable sets $F_1, F_2, \ldots \in \cF_0$ with
$\mu (F_i) < \infty$ and $\setX = F_1 \cup F_2 \cup \cdots$ 
\citep[p.5010]{TaraldsenLindqvist16renyi}.
A measure space $(\setX, \cF, \mu)$ is sigma-finite if
$\cF$ is sigma-finite,
and $\mu$ is then also said to be sigma-finite
\citep[p.121]{RUDIN}. 

\subsection{Conditional probability space}

A bunch $\cB$ in a measurable space is a family of
measurable sets closed under finite unions that does not contain
the empty set, but contains a countable family $F_1, F_2, \ldots$ of sets whose union is the whole set \citep[p.38]{RENYI}.
A bunch $\cB$ is ordered if $B_1,B_2 \in \cB$ implies
$B_1 \subset B_2$ or $B_2 \subset B_1$.

\subsection{Statistical model}

A statistical model is a triple $(\Omega, X, \Theta)$
where the space $\Omega$ is a conditional probability space,
the data $X$ is a measurable function
$X: \Omega \into \Omega_X$, 
and the model parameter
$\Theta$ is a measurable function 
$\Theta: \Omega \into \Omega_\Theta$.
These definitions,
and the ones that follow,
are as given by \citet{SCHERVISH} except for choice of symbols
and the generalization given by assuming
that $\Omega$ is a conditional probability space.
There is one probability law $\pr$ defined on the sigma-algebra
$\cE$ of events in $\Omega$ - and all other concepts
are defined from the basic space $\Omega$
\citep[p.5011]{TaraldsenLindqvist16renyi}.

A statistic 
$Y = \phi(X) = \phi \circ X$ 
is a measurable function of the data 
and a parameter 
$\Gamma = \psi(\Theta) = \psi \circ \Theta$
is a measurable function
of the model parameter as
illustrated in equation~(\ref{eqStatMod})
\be{StatMod}
\scalebox{1.2}{
\begin{tikzcd}[row sep=normal, ampersand replacement=\&]
\&\Omega_\Theta \arrow[r, "\psi"] \& \Omega_\Gamma  \\
(\Omega, {\cal E}, \pr) \arrow[ur, "\Theta"] 
\arrow[drr, "Y" near end] \arrow[dr, "X"'] 
\arrow[urr, "\Gamma"' near end] \& \& \\
\&\Omega_X \arrow[r, "\phi"'] \& \Omega_Y
\end{tikzcd}
}
\ee
A random quantity is
a measurable function $Z: \Omega \into \Omega_Z$,
and its law is defined by 
$\pr_Z (A) = \pr (Z \in A)$ where 
$(Z \in A) = \{\omega \st Z(\omega) \in A\}$.
We abuse notation here and interpret $\pr$
as one fixed representative of the equivalence class that
defines $\pr$ as a conditional measure. 
A random quantity is sigma-finite if its law is
sigma-finite. 
If the model parameter $\Theta$ is sigma-finte, 
then the conditional probabilities
$\pr^\theta_X (A) = \pr (X \in A \st \Theta = \theta)$
define a family of probability measures on
the sample space $\Omega_X$ indexed by
the model parameter $\theta$ in the 
model parameter space $\Omega_\Theta$. 
Likewise, if the data $X$ is sigma-finite,
then the posterior 
$\pr^x_\Theta (B) = \pr (\Theta \in B \st X = x)$
is a probability measure
on the model parameter space $\Omega_\Theta$.
The mappings
$\theta \mapsto \pr_X^\theta (A)$ and
$x \mapsto \pr_\Theta^x (B)$ are measurable for all events $A$,
but existence of families of probability measures
as claimed above requires
regularity assumptions:
It is sufficient to assume that
the sample space $\Omega_X$ and the model parameter space $\Omega_\Theta$ 
are Borel spaces \citep[p.619]{SCHERVISH}\citep[p.5011]{TaraldsenLindqvist16renyi}.


} 

\bibliographystyle{chicago}

\bibliography{bib,JABREFgtaralds}


\end{document}